\documentclass[a4paper,11pt,oneside]{amsart}
\usepackage[left=2.7cm,right=2.7cm,top=3.5cm,bottom=3cm]{geometry}

\usepackage[all]{xy}
\usepackage{amscd,mathrsfs}
\usepackage{amsfonts}
\usepackage{amssymb}
\usepackage{amsmath}
\usepackage{latexsym}
\usepackage[usenames,dvipsnames]{color}
\usepackage{verbatim}

\newcommand{\C}{\mathbb{C}}
\newcommand{\F}{\mathbb{F}}

\newcommand{\N}{\mathbb{N}}
\newcommand{\Q}{\mathbb{Q}}

\newcommand{\Z}{\mathbb{Z}}
\newcommand{\bbS}{\mathbb{S}}

\newcommand{\g}{\gamma}

\newcommand{\calc}{\mathcal C}
\newcommand{\calf}{\mathcal F}

\newcommand{\calm}{\mathcal M}

\newcommand{\calCl}{\mathcal{C}\ell^{\,0}}

\newcommand{\pr}{\mathfrak p}

\newcommand{\gotq}{\mathfrak q}

\newcommand{\sri}{\twoheadrightarrow}
\newcommand{\iri}{\hookrightarrow}

\newcommand{\ov}{\overline}
\newcommand{\wt}{\widetilde}

\newcommand{\plim}[1]{\displaystyle{\lim_{\stackrel{\longleftarrow}{#1}}}\,}

\DeclareMathOperator{\Gal}{Gal} \DeclareMathOperator{\Hom}{Hom}
\DeclareMathOperator{\Fitt}{Fitt} 
 
\DeclareMathOperator{\Fr}{Fr}

\theoremstyle{plain}
\newtheorem{thm}{Theorem}[section]
\newtheorem{cor}[thm]{Corollary}
\newtheorem{lem}[thm]{Lemma}

\theoremstyle{definition}
\newtheorem{defin}[thm]{Definition}
\newtheorem{rem}[thm]{Remark}

\title[Fitting ideals of Carlitz-Hayes cyclotomic type extensions]
{Fitting ideals of class groups in Carlitz-Hayes
cyclotomic extensions}

\author[A. Bandini] {Andrea Bandini}
\address{Andrea Bandini: Dipartimento di Matematica, Universit\`a di Pisa\\
 Largo Bruno Pontecorvo 5, 56127 Pisa, Italy}
\email{andrea.bandini@unipi.it}

\author[F. Bars] {Francesc Bars}
\address{Francesc Bars: Departament de Matem\`atiques, Facultat de Ciencies, Universitat Aut\`onoma de Barcelona\\
08193 Bellaterra (Barcelona), Catalonia}
\email{francesc@mat.uab.cat}
\thanks{F. Bars supported by MTM2016-75980-P and MDM-2014-0445.}

\author[E. Coscelli]{Edoardo Coscelli}
\address{E.Coscelli: Dipartimento di Matematica F. Enriques, Universit\`a degli Studi di Milano\\
 Via Cesare Saldini, 50, 20133 Milano (MI), Italy} \email{edoardo.coscelli@unimi.it}

\begin{document}

\maketitle

\begin{flushright}
    \textsc{Dedicated to the memory of David Goss}
  \end{flushright}

\begin{abstract} We generalize some results of Greither and Popescu
to a geometric Galois cover
$X\rightarrow Y$ which appears naturally for example in extensions generated by
$\mathfrak{p}^n$-torsion points of a rank 1 normalized
Drinfeld module (i.e. in subextensions of Carlitz-Hayes
cyclotomic extensions of global fields of positive characteristic).
We obtain a description of the Fitting ideal of class groups (or of their dual) via a 
formula involving Stickelberger elements and providing a link (similar to the one in \cite{ABBL}) 
with Goss $\zeta$-function.

\end{abstract}



\maketitle

\section{Introduction}\label{SecIntro}
One of the main topics of modern number theory is the investigation of arithmetic
properties of motives over a global field (in any characteristic) and their relation with (or interpretation as) special values of
$\zeta$-functions or $L$-functions. Iwasawa theory offers an effective way of dealing with various issues arising in this context: on one side it 
deals with the variation of algebraic structures in $p$-adic towers (e.g. class groups) and, on the other side, it provides a good understanding 
of special values via interpolation with $p$-adic $L$-functions (whose construction is one of the major outcome of the theory).
The algebraic counterpart of $p$-adic $L$-functions is usually represented by (generators of) characteristic ideals or Fitting ideals
and the link between the algebraic and analityc side of the theory is the subject of various instances of Iwasawa Main Conjecture
(IMC).

Let $F$ be a function field of transcendence degree 1 over a finite field $\F$ of cardinality $q=p^r$, i.e. the function field of a smooth projective 
curve $Y$ defined over $\F$. For the particular case of $F=\F(t)$ and the cyclotomic $\Z_p^\infty$-extension $\mathcal{F}_\pr/F$ 
generated by the $\pr^n$-powers torsion of the Carlitz module ($\pr$ any prime of $\F[t]$), the IMC is proved in \cite{ABBL}: the authors
use some results of Greither and Popescu on cohomological triviality of $p$-adic motives and on Fitting ideals (see \cite{GP1} and \cite{GP2}) 
to compute Fitting ideals of class groups of the finite subextensions of $\mathcal{F}_\pr$, and then check all compatibility conditions needed
to apply a limit process (similar results for the same type of $\Z_p^\infty$-extensions but for more general function fields $F$ can be 
found in the third author PhD thesis \cite{CPhD}). The main goal of this paper is to describe a more general setting: we consider Galois
extensions (or Galois coverings between curves) with at least one totally ramified prime (or {\em almost} totally ramified,
see Section \ref{SecAlmostTotRam}) and provide a general formula for the Fitting ideal of (the Pontrjagin dual of) their class groups applying 
again the result of \cite{GP1} and \cite{GP2}. The main application we have in mind is the following: let $F$ be a global function field as above
and consider its extension generated by the $\pr^n$- torsion of a sign normalized rank 1 Drinfeld module (or a {\em Hayes module}) $\Phi$.
The extension $F(\Phi[\pr])/F$ contains the Hilbert class field $H_A$ of $F$ ($A$ is the ring of integers of $F$ with respect to a fixed prime 
at $\infty$) and its Galois group acts on all class groups of the extensions $F(\Phi[\pr^n])$. Therefore we can consider $\chi$-parts of 
those class groups for any character $\chi$ of $Gal(F(\Phi[\pr])/F)$ and the computation of their Fitting ideals should lead to special values
(and sometimes trivial zeros) of a $\pr$-adic $L$-function or of Goss $\zeta$-function for $F$. Unfortunately our computations strongly depend on the
behavior of the character $\chi$ (i.e. on $Ker(\chi)$) and in some cases it is only possible to compute Fitting ideals of duals of class groups:
nevertheless the relations/formulas proved in Theorems \ref{teo3.2} and \ref{teo3.3} (and in their corollaries) provide the algebraic side
of the IMC for this setting (see also \cite{CPhD}) and shed some light on the phenomenon of double zeroes for type 3 characters (see Section
\ref{SecFittIdCycExt}) which did not appear in the basic setting $F=\F(t)$ of \cite{ABBL}.    

\subsection{Brief summary} We start by considering the following general setting: let $X\rightarrow Y$ be a finite abelian Galois covering 
of smooth projective curves over a finite field $\F$ as above, with Galois group $G$ and unramified outside a finite set of places $S$. 
Equivalently we can consider a finite abelian Galois extensions $K/F$ (where $K$ and $F$ are the function fields $\F(X)$ and $\F(Y)$ 
respectively) with Galois group $G$. We always assume that $\F$ is the field of constants of $X$, i.e. the covering $X\rightarrow Y$ is geometric. 
We generalize some of the results of Greither and Popescu in \cite[Sections 2 and 3]{GP2} (with $\ell=p$) on the Fitting ideal of the 
Tate module of the Picard 1-motive of $X$ to the case in which there is a totally ramified place $v\in S$ (not necessarily $\F$-rational). 
We use this to compute Fitting ideals of $\Z_p$-duals of Tate modules which are easily linked  to Pontrjagin duals of class groups of 
$\F(X)$ (see Theorems \ref{teoap2} and \ref{teoap4}). Then, in Section \ref{SecAlmostTotRam}, we obtain similar results for a 
mixed cover, i.e. $X\rightarrow Y'\rightarrow Y$ where $X\rightarrow Y'$ is of $p$-power degree and totally ramified at some prime 
and $Y'\rightarrow Y$ is of order prime to $p$.

In Section \ref{SecFittIdCycExt} we specialize to the cyclotomic extension generated by the $\pr^n$-powers torsion of an Hayes module, where
$F_n:=F(\Phi[\pr^n])$ plays the role of $\F(X)$ (for some $n\geqslant 2$) and $F_1:=F(\Phi[\pr])$ plays the role of $\F(Y')$ and we assume that
the class group of $F=\F(Y)$ has order prime to $p$.  The main outcome is summarized in Theorem \ref{teo3.3} and Corollary \ref{CorDualClGr}
where we provide formulas for Fitting ideals for class groups and duals of class groups involving Stickelberger elements (which arise
from the main theorems of Greither and Popescu) and some correction factors like $\frac{1}{1-\gamma^{-1}}$ (where $\gamma$
is a topological generator of $Gal(\overline{\F}/\F)\,$) which emphasize the presence of trivial zeroes (sometimes of order 2). 
For example let $W$ be a ring extension of $\Z_p$ containing all values of characters defined over $Gal(F_1/F)$, let $\Theta_{F_n/F,S,\chi}$ 
be the Stickelberger element for the extension $F_n/F$ (more details are in Section \ref{SecFittId}, here $S:=\{\pr,\infty\}$), put $d_\pr:=\deg(\pr)$ and 
$n(Gal(F_n/F_1)):=\displaystyle{\sum_{\sigma\in Gal(F_n/F_1)} \sigma}$. If $\chi$ is trivial on $Gal(F_1/H_A)$ (i.e. it is 
{\em of type 3} in the terminology od Definition \ref{DefCharType}) and on the Frobenius at $\pr$ but it is not the trivial character $\chi_0$, then
\begin{equation}\label{EqIntro1} \Fitt_{W[Gal(F_n/F_1)]} (\calCl(F_n)(\chi)^{\vee})=
\frac{\Theta_{F_n/F,S,\chi}}{(1-\gamma^{-1})^2}_{|\gamma=1}\cdot\left(\frac{n(Gal(F_n/F_1))}{d_{\pr}}\right) .
\end{equation}

In the final Section \ref{SecGoss} we point out some possible applications of our results to the study of special values of Goss
$\zeta$-function.

\subsection{Basic notations}  
Let $F$ be a global function field of positive characteristic $p$, i.e. a finite extension of a field of transcendence degree 1 over a finite
field $\F:=\F_q$, which we call the constant field of $F$. We fix a place $\infty$ of $F$ and let $A$ be the ring of integers at $\infty$,
i.e. the elements of $F$
which are regular outside $\infty$. We denote by $H_A$ the Hilbert class field of $A$, i.e. the maximal abelian extension of $F$ 
which is unramified at every prime of $A$ and totally split at $\infty$. Let $h^0(F)$ be the class number of $F$ so that $[H_A:F]=h^0(F)$.
We assume that $p\nmid h^0(F)$, moreover we assume also that the degree of $\infty$ is $\deg(\infty):=d_\infty=1$ so that the 
constant field of $H_A$ is still $\F$.

\begin{subsection}*{Acknowledgements} We (as probably most of the people actually working on function fields) are deeply indebted
to the efforts made by David Goss throughout his career to put function fields on the (so to speak) matematical map.
He never missed an opportunity to encourage young mathematicians to explore this beautiful and still quite mysterious world and
his legacy goes much beyond his remarkable work.\\
The second author wishes to recall in particular the enthusiasm and energy that
David was able to pass on. When he met him for the first time in
Barcelona during 2010, he was totally fascinated by his energy,
his positivity and his dedication to the area of function fields arithmetic.\\
Thank you David.
\end{subsection}

\begin{section}{Fitting ideals: totally ramified and almost totally ramified cases}\label{SecFittId}
We start by fixing a few notations for the main objects we will be interested in and by recalling the general setting
already mentioned in the introduction. Let $X\rightarrow Y$ be a finite abelian Galois covering of smooth projective curves over a finite field 
$\F$ (of order $q=p^r$) with Galois group $G$ and unramified outside a finite set of places $S$ (i.e. in terms of function fields,
$G=Gal(\F(X)/\F(Y))$\,). We assume that $X\rightarrow Y$ is geometric, i.e. $\F$ is the field of constants of $X$, and that $S\neq\emptyset$.

\noindent Let $\overline{S}:=\overline{X}(S)$ be the set of closed points of the base change $\overline{X}:=X\times_\F \overline{\F}$
(where $\overline{\F}$ is a fixed algebraic closure of $\F$) lying above points of $S$, and let $\Sigma$ be a set of primes of $F:=\F(Y)$ disjoint from $S$.
Let $T_p(\calm_{\ov{S},\ov{\Sigma}})$ be the $p$-adic realization (or $p$-adic Tate module) associated with the Picard 1-motive 
$M_{\ov{S},\ov{\Sigma}}$: a detailed description of $\calm_{\ov{S},\ov{\Sigma}}$ in terms of divisor classes quite useful for 
computations is provided in \cite[Section 2]{GP1}. For any set $\Sigma$ there is a short exact sequence (see \cite[Equation (2)]{GP2})
\[ 0\rightarrow T_p(\tau_{\Sigma}(\ov{\F}))\rightarrow T_p(\calm_{\ov{S},\ov{\Sigma}})\rightarrow
T_p(\calm_{\ov{S}})\rightarrow 0 \ ,\] 
and, since we will only consider the $p=char(F)$ case, the toric part $T_p(\tau_{\Sigma}(\ov{\F}))$ vanishes (see \cite[Remark 2.7]{GP1}). Hence
$\Sigma$ has no concrete influence on the module $T_p(\calm_{\ov{S},\ov{\Sigma}})\simeq T_p(\calm_{\ov{S}})$ we are interested
in and we can (and will) assume $\Sigma=\emptyset$.

\begin{defin}\label{DefStick}
Let $K:=\F(X)$ and define the {\em Stickelberger series} associated with
$K/F$ (or $X\rightarrow Y$) and $S$ as
\[ \Theta_{K/F,S}(u):= \prod_{\gotq\not\in S} (1-\Fr_\gotq^{-1} u^{\deg(\gotq)})^{-1} \in \Z[G][[u]] ,\]
where the product is taken over places $\gotq$ of $F$ and $\Fr_\gotq\in G$ denotes the Frobenius at $\gotq$.
\end{defin}

One can actually define a kind of {\em universal} Stickelberger series in $\Z[G_S][[u]]$, where $G_S$ is the Galois group of the maximal
abelian extension of $F$ unramified outside $S$, and find back $\Theta_{K/F,S}(u)$ as the natural projection of that series,
(see \cite[Section 3.1]{ABBL}).

In this setting the main result of \cite{GP1} (i.e. \cite[Theorem 4.3]{GP1}, see also \cite[Lemma 2.3]{GP2}) reads as follows

\begin{thm}[Greither-Popescu]\label{teoap1} The Tate module $T_p(\calm_{\ov{S}})$ is cohomologically trivial
over $G$ and free of finite rank over $\Z_p$, hence projective over $\Z_p[G]$. Moreover the Fitting ideal of
$T_p(\calm_{\ov{S}})$ over $\Z_p[G][[G_\F]]$ is principal and generated by $\Theta_{K/F,S}(\g^{-1})$, i.e.
\begin{equation}\label{EgGP1}
\Fitt_{\Z_p[G][[G_\F]]} (T_p(\calm_{\ov{S}}))=(\Theta_{K/F,S}(\g^{-1}))
\end{equation}
(where $\g$ is the arithmetic Frobenius in $G_\F:=\Gal(\ov{\F}/\F)$\,).
\end{thm}

\subsection{The totally ramified case}
In this subsection we assume that the cover $X\rightarrow Y$ is {\em totally ramified} at some place $v_1\in S$.

We recall the short exact sequence of
$\Z_p[G][[G_{\F}]]$-modules 
\begin{equation}\label{eqap0}
0\rightarrow T_p(Jac(X)(\overline{\F}))\rightarrow T_p(\calm_{\ov{S}})\rightarrow L\rightarrow 0 ,
\end{equation}
(see \cite[after Definition 2.6]{GP1}) where $Jac(X)(\ov{\F})$ is the set of the $\ov{\F}$-points of the Jacobian of $X$, $T_p(Jac(X)(\ov{\F}))$ 
is the usual Tate module at $p$ and $L$ is the kernel of the degree map $\Z_p[\ov{X}(S)]\rightarrow \Z_p$. 
Since our first goal is to compute the Fitting ideal of $T_p(Jac(X)(\overline{\F}))$ and Theorem \ref{teoap1} takes care of the central 
element in the sequence, we focus now on the $\Z_p[G][[G_\F]]$-module $L$.  

\noindent For a place $v\in Y$ denote by $G_v$ the decomposition subgroup of $v$ in $G$ and by $I_v$ its
inertia subgroup. Put $H_v:=\Z_p[\ov{X}(v)]$ (where $\ov{X}(v)$ denotes the set of points of $X\times_{\F}\overline{\F}$ above $v$):
it is a $\Z_p[G][[G_\F]]$-module and we observe that $\Z_p[\ov{X}(S)]=\oplus_{v\in S} H_v$. Let $\Fr_v$ denote the Frobenius
of $v$ in $G$ (if $v$ is unramified, in the ramified case any lift of a Frobenius of $v$ in $G_v/I_v$ will do) and put
$e_v(u):=1-\Fr_v^{-1}u^{d_v}\in\Z[G][u]$ (where $d_v$ is the degree of the place $v$).

\noindent In this setting \cite[Lemmas 2.1 and 2.2]{GP2} read as

\begin{lem} \label{LemGP}
$H_v$ is a cyclic $\Z_p[G][[G_\F]]$-module and we have:
\begin{enumerate}
\item[(i)] if $v\in Y$ is unramified in $X$, then
\[ \Fitt_{\Z_p[G][[G_\F]]}(H_v)=(e_v(\g^{-1}))\quad {\rm and}\quad H_v\simeq \Z_p[G][[G_\F]]/(e_v(\g^{-1}))\ ;\]
\item[(ii)] if $v\in Y$ is ramified in $X$, then
$\Fitt_{\Z_p[G][[G_\F]]}(H_v)=(e_v(\g^{-1}),\tau-1\,:\,\tau\in I_v)$
and
\[ H_v\simeq \Z_p[G][[G_\F]]/(e_v(\g^{-1}),\tau-1\,:\,\tau\in I_v)\simeq \Z_p[G/I_v][[G_\F]]/(e_v(\g^{-1})) \ .\]
\end{enumerate}
\end{lem}

\begin{lem}\label{lemA3} Let $X\rightarrow Y$ be a geometric Galois cover with a totally ramified prime $v_1\in S$. Then
 we have an isomorphism of $\Z_p[G][[G_{{\F}}]]$-modules:
\begin{equation}\label{eqap1}
L\simeq (\g-1)H_{v_1}\oplus\left(\bigoplus_{v\in S'} H_v\right)
\end{equation}
where $S':=S-\{v_1\}$.
\end{lem}

\begin{proof} The primes in $\ov{X}(S)$ are points in $X\times_{\F}\ov{\F}$ and have degree 1 so $\deg : H_v\rightarrow\Z_p$ is surjective.
Since $v_1$ is totally ramified, the degree map provides a decomposition $H_{v_1}=(\gamma-1)H_{v_1}+ \Z_p {\bf 1}_{H_{v_1}}$ 
(where ${\bf 1}_{H_{v_1}}$ denotes the unit element of $H_{v_1}$), moreover the previous lemma yields 
$H_{v_1}\simeq \Z_p[[G_\F]]/(1-\g^{-d_{v_1}})$ as a $\Z_p[G][[G_\F]]$-module.

We have an injective morphism of $\Z_p[G][[G_{\F}]]$-modules
\[ (\g-1)H_{v_1}\oplus(\oplus_{v\in S'} H_v)\hookrightarrow ((\gamma-1)H_{v_1}+ \Z_p {\bf 1}_{H_{v_1}})\oplus\left(\bigoplus_{v\in S'} H_v)\right) \] 
given by
\[ (\alpha,\beta)\mapsto (\alpha-deg(\beta){\bf 1}_{H_{v_1}},\beta) \,.\]
Since all points of $\ov{X}(v_1)$ have degree 1, the degree map on $H_{v_1}$ (via the identification with
$\Z_p[[G_\F]]/(1-\g^{-d_{v_1}})\,$) sends ${\bf 1}_{H_{v_1}}$ to a unit in $\Z_p$. Hence the image of the 
morphism is inside $L$, and it is actually equal to $L$ because of the above decomposition for
$H_{v_1}$ (one can also check directly that the $\Z_p$-ranks are the same).
\end{proof}

Now we can generalize \cite[Theorem 2.6]{GP2}.

\begin{thm}\label{teoap2} Assume $X\rightarrow Y$ a geometric Galois cover with a totally ramified point $v_1\in S$ of degree $d_{v_1}$. 
Then
\begin{equation}\label{AppFitt}
\Fitt_{\Z_p[G][[G_\F]]}(T_p(Jac(X)(\ov{\F}))^*)=(\Theta_{K/F,S}(\g^{-1}))\cdot
\left(1,\frac{n(G)}{\frac{1-\g^{-d_{v_1}}}{1-\g^{-1}}}\right)\cdot
\prod_{v\in S' } \left(1,\frac{n(I_v)}{e_v(\g^{-1})}\right)\ ,
\end{equation}
where $^*$ denotes the $\Z_p$-dual and, for any group $N$, we put $n(N):=\sum_{\sigma\in N}\sigma\in\Z[N]$.
\end{thm}

\begin{proof} By Lemma \ref{lemA3}, $L\simeq (\g-1)H_{v_1}\oplus\left(\oplus_{v\in S'}H_v\right)$ as a
$\Z_p[G][[G_\F]]$-module. For any $v\in S$ we have the following short exact sequence (compare with \cite{GP2} after Lemma 2.4)
\begin{equation}\label{eqap2}
0\rightarrow H_v\rightarrow \Z_p[G][[G_{\F}]]/e_{v}(\g^{-1})\rightarrow \Z_p[G][[G_{\F}]]/(e_v(\g^{-1}),n(I_v))\rightarrow 0
\end{equation}
(where the map on the left sends ${\bf 1}_{H_v}$ to $n(I_v)$\,).
 We now obtain a short exact sequence for the $\Z_p[G][[G_\F]]$-module $(\g-1)H_{v_1}$ (which is not trivial whenever $d_{v_1}>1$). 
Write
\[ (\g-1)H_{v_1}=(\g -1)\Z_p{\bf 1}_{v_1}+\cdots+(\g^{d_{v_1}-1}-\g^{d_{v_1}-2})\Z_p{\bf 1}_{v_1} \]
as a free $\Z_p[G]$-module of rank $d_{v_1}-1$, and put $w_{i}:=(\g^i-\g^{i-1}){\bf 1}_{v_1}$. Fixing the basis
$w_1,\ldots,w_{d_{v_1}-1}$ and applying \cite[Proposition 2.1]{GP1} we compute
\[ \det({\bf Id} - \g u\,|\,(\g-1)H_{v_1})=1+u+u^2+\ldots+u^{d_{v_1}-1}=\frac{1-u^{d_{v_1}}}{1-u} \,.\]
Therefore one has a short exact sequence
\begin{equation}\label{eqap3}
(\g-1)H_{v_1} \iri
\Z_p[G][[G_\F]]/\left(\frac{1-\g^{-d_{v_1}}}{1-\g^{-1}}\right) \sri
\Z_p[G][[G_\F]]/\left(\frac{1-\g^{-d_{v_1}}}{1-\g^{-1}},n(G)\right)\,.
\end{equation}
Putting together equations \eqref{eqap0}, \eqref{eqap2} and \eqref{eqap3} we obtain the four term exact
sequence
\[ T_p(Jac(X)(\ov{\F})) \iri T_p(\calm_{\ov{S}}) \rightarrow
\bigoplus_{v\in S'}\Z_p[G][[G_\F]]/e_v(\g^{-1})\oplus
\Z_p[G][[G_\F]]/\left(\frac{1-\g^{-d_{v_1}}}{1-\g^{-1}}\right) \]
\vspace{-1truecm}
\[ \hspace{4.5truecm} \xymatrix{ \ar@{->>}[d] \\ \ } \]
\vspace{-.5truecm}
\[ \hspace{3truecm} \bigoplus_{v\in S'}\Z_p[G][[G_\F]]/(e_v(\g^{-1}),n(I_v))\oplus
\Z_p[G][[G_\F]]/\left(\frac{1-\g^{-d_{v_1}}}{1-\g^{-1}},n(G)\right) .\] 
Denote by $X_2$, $X_3$ and $X_4$ the second, third and fourth modules appearing in the sequence above. All modules are finitely
generated and free over $\Z_p$, moreover $X_3$ has projective dimension 0 or 1 over $\Z_p[G][[G_\F]]$, while $X_2$ has projective
dimension 1 over $\Z_p[G][[G_\F]]$ because it has no non-trivial finite submodules (enough by \cite[Proposition 2.2 and Lemma
2.3]{Po}) and is $G$-cohomologically trivial by Theorem \ref{teoap1}. With these properties \cite[Lemma 2.4]{GP2} yields
\[ \Fitt_{\Z_p[G][[G_\F]]}(T_p(Jac(X)(\ov{\F}))^*)\Fitt_{\Z_p[G][[G_\F]]}(X_3) \!\! = \!\!
\Fitt_{\Z_p[G][[G_\F]]}(X_2)\Fitt_{\Z_p[G][[G_\F]]}(X_4).\] Since
\[ \Fitt_{\Z_p[G][[G_\F]]}(X_3)=\left(\frac{1-\g^{-d_{v_1}}}{1-\g^{-1}}\right)\cdot\left(\prod_{v\in S'}e_v(\g^{-1})\right) \]
and
\[ \Fitt_{\Z_p[G][[G_\F]]}(X_4)=\left(\frac{1-\g^{-d_{v_1}}}{1-\g^{-1}},n(G)\right)\cdot\prod_{v\in S'}(e_v(\g^{-1}),n(I_v)) \ ,\]
Theorem \ref{teoap1} immediately implies equation \eqref{AppFitt}.\end{proof}

To apply this to class groups just note that
\[ T_p(Jac(X)(\ov{\F}))^*\simeq \Hom(Jac(X)(\ov{\F})\{ p\},\Q_p/\Z_p) \]
where $Jac(X)(\ov{\F})\{ p\}$ denotes the divisible group given by the $p$-power torsion elements and $^*$ indicates the $\Z_p$-dual.
Consider the projection morphism $\pi^{G_\F}:\Z_p[G][[G_\F]]\rightarrow\Z_p[G]$ mapping $\g$ to 1.
For a finitely generated $\Z_p[G][[G_\F]]$-module $M$, we have the equality
\[ \pi^{G_\F}\left(\Fitt_{\Z_p[G][[G_\F]]}(M)\right)=\Fitt_{\Z_p[G]}(M_{G_\F}) ,\]
where $M_{G_\F}$ denotes the $G_\F$-coinvariants $M/(\gamma-1)M$.
Moreover the $G_\F$-coinvariants of $T_p(Jac(X)(\ov{\F}))^*$ correspond to
\[ \begin{array}{ll} (T_p(Jac(X)(\ov{\F}))^*)_{G_\F} & \simeq\Hom (Jac(X)(\ov{\F})^{G_\F},\Q_p/\Z_p) \\
\ & = \Hom(Jac(X)(\F),\Q_p/\Z_p)=\calCl(X)^\vee  \,,\end{array}\] 
where the final module is the Pontrjagin dual of the $p$-part of the class group associated with 
$X$ (see \cite[Lemma 4.6 and Remark 4.7]{ABBL}).

The following generalizes \cite[Theorem 3.2]{GP2}.

\begin{thm}\label{teoap4}
Assume $X\rightarrow Y$ is a geometric Galois cover with a totally ramified point $v_1\in S$. Then
\[ \Fitt_{\Z_p[G]}(\calCl(X)^\vee)=\langle\, g_{\mathcal{W}} \cdot
cor^G_{G/I_{\mathcal{W}}}(\Theta_{K^{\mathcal{W}}/F,S - \mathcal{W}}(1)) \,,\,\mathcal{W}\subset S'\,\rangle\,,\] 
where $cor$ denotes the corestriction map, and for any $T\subset S$ we write $I_{T}$ for the compositum of all inertia groups $I_v$ with
$v\in {T}$, $K^{T}:=K^{I_{T}}$ and $g_{T}=\displaystyle{\frac{\prod_{v\in T}|I_v|}{|\prod_{v\in T} I_v|}}\in\N$.
\end{thm}

\begin{proof}
As explained in \cite[p. 232]{GP2}, the Euler relations give us all the generators for the Fitting ideal (using Theorem \ref{teoap2}), and
we have for any subset $T$ of $S$
\[ \left(\prod_{v\in T}n(I_v)\right)\Theta_{K/F,S}(\g^{-1})=g_T\cdot \prod_{w\in T} e_v(\g^{-1})\cdot
cor^G_{G/I_T}(\Theta_{K^T/F,S-T}(\g^{-1})) \,.\] 
Moreover, if $v_1\notin T$ we obtain the relations that appear in Theorem \ref{teoap2} while, if $v_1\in T$, then
\[ \frac{n(G)}{e_{v_1}(\g^{-1})} \prod_{v\in T-v_1}\frac{n(I_v)}{e_v(\g^{-1})} \Theta_{K/F,S}(\g^{-1})
= g_T\cdot cor^G_{G/I_T} (\Theta_{K^T/F,S-T}(\g^{-1})) \,.\] 
Note that $e_{v_1}(\g^{-1})=1-\g^{-d_{v_1}}$ and consider the element
\[ e_{d_{v_1}}:=\frac{e_{v_1}(\g^{-1})}{\frac{1-\g^{-d_{v_1}}}{1-\g^{-1}}}= 1-\g^{-1}\in\Z_p[G][[G_\F]] \ ,\]
which is an unit of $\Z_p[G][[G_\F]]$. Multiplying the equalities above by $e_{d_{v_1}}$ we obtain equalities of ideals in
$\Z_p[G][[G_\F]]$.

\noindent The projection map $\pi^{G_\F}$ maps $e_{d_{v_1}}$ to zero and, using Theorem \ref{teoap2}, we obtain the claim (in particular
no more generators are needed for the Fitting ideal over $\Z_p[G]$ when $v_1\in T$).\end{proof}

\subsection{The almost totally ramified case}\label{SecAlmostTotRam} 
In this subsection, we assume that $X\rightarrow Y$ is a finite abelian geometric cover with Galois group $\widetilde{G}$, ramified at a
finite set $S$ and such that it factors through $X\rightarrow Y'\rightarrow Y$ (with $X\neq Y'$), where $\F(X)/\F(Y')$ is a $p$-extension totally 
ramified at some place $v'_1$ of $Y'$ lying above a prime $v_1$ of $Y$ and $\F(Y')/\F(Y)$ is a Galois extension of degree coprime with $p$
(i.e. $Gal(\F(X)/\F(Y'))$ is the $p$-Sylow subgroup of $\widetilde{G}$) \footnote{The fact that $p\nmid [\F(Y'):\F(Y)]$ is essential. But if there is
subextension in $\F(X)/\F(Y')$ which has degree prime to $p$ (and is totally ramified) one can move it to $\F(Y')/\F(Y)$ by enlarging
$\F(Y')$, so the fact that $\F(X)/F(Y')$ is a $p$-extension is not really restrictive.}.
We put  $\widetilde{G}\simeq Gal(\F(X)/\F(Y'))\times Gal(\F(Y')/\F(Y)):=G\times H$ 
and consider $\chi\in Hom(H,\mathbb{C}^*)$ a character of $H$. 
All such characters have values in $\mu_{|H|}$ so we fix a ring extension $W$ of $\Z_p$ containing $\mu_{|H|}$, and consider 
$\chi$ as a $p$-adic character with values in $W^*$. As usual we denote by $e_\chi$ the idempotent associated with $\chi$, i.e.
\[ e_{\chi}:=\frac{1}{|H|}\sum_{\delta\in H}\chi(\delta^{-1})\delta\in W[H] \,,\] 
and, for any $\Z_p[\wt{G}][[G_\F]]$-module $M$ we write $M(\chi)$ for the $\chi$-part of $M$ (i.e. the submodule $e_{\chi}(M\otimes_{\Z_p}W)$\,),
which is a $W[G][[G_{\mathbb{F}}]]$-module inside $M\otimes_{\Z_p}W$.
The trivial character will be denoted by $\chi_0$ as usual.

We identify $\widetilde{G}$ with $G\times H$ and denote by $\pi_H$ (resp. $\pi_G$) the canonical projection 
$W[\widetilde{G}]\rightarrow W[H]$ (resp. $W[\widetilde{G}]\rightarrow W[G]$).

In this setting we can take $\chi$-parts in the exact sequence (\ref{eqap0}) obtaining
\begin{equation}\label{EqChiPart} 
0\rightarrow T_p(Jac(X)(\overline{\F}))(\chi)\rightarrow T_p(\mathcal{M}_{\ov{S}})(\chi)\rightarrow L(\chi)\rightarrow 0\,,
\end{equation} 
where $L(\chi)$ is the ($\chi$-part of the) kernel of degree map $W[\ov{X}(S)]\rightarrow W$ (and we consider the trivial action on $W$).

\begin{lem}\label{lem2.3gen} 
Let $X\rightarrow Y$ be a geometric abelian Galois cover as above. Then we have an isomorphism of $W[G][[G_{\F}]]$-modules 
\[ L(\chi_0)\simeq (\gamma-1)H_{v_1}(\chi_0)\oplus \left(\bigoplus_{v\in S'} H_v(\chi_0) \right) \] 
where $S'=S-\{v_1\}$, and
\[ L(\chi)= \bigoplus_{v\in S} H_v(\chi)\quad for\ any\ \chi\neq\chi_0\,.\]
\end{lem}

\begin{proof} Consider the degree map $\deg:\oplus_{v\in S} H_v\rightarrow \Z_p$. Taking $\chi$-parts for $\chi\neq \chi_0$ one immediately
has $\deg(\chi):\oplus_{v\in S} H_v(\chi)\rightarrow 0$, hence $L(\chi)= \oplus_{v\in S} H_v(\chi)$.

Now we deal the case $\chi=\chi_0$: the hypothesis on the ramification of $v_1$ yields
\[ H_{v_1}=\Z_p[\widetilde{G}/I_{v_1}][[G_{\F}]]/(e_{v_1}(\gamma^{-1}))=
\Z_p[H/\pi_{H}(I_{v_1})][[G_{\F}]]/(e_{v_1}(\gamma^{-1})) \] 
and, since $\chi_0$ is trivial on $\pi_H(I_{v_1})$ and on $\Fr_{v_1}$, we have
$H_{v_1}(\chi_0)\simeq W[[G_{\F}]]/(1-\gamma^{-d_{v_1}})$ as a $W[G][[G_{\F}]]$-module.

We have an injective morphism of $W[G][[G_{\F}]]$-modules
\[ (\gamma-1)H_{v_1}(\chi_0)\oplus\left(\bigoplus_{v\in S'}H_v(\chi_0)\right)\hookrightarrow ((\gamma-1)H_{v_1}(\chi_0) +
W{\bf 1}_{H_{v_1}(\chi_0)})\oplus \left(\bigoplus_{v\in S'}H_v(\chi_0)\right) \] 
defined exactly as in Lemma \ref{lemA3} and the proof follows the same path.
\end{proof}

\begin{lem}\label{lem2.4modchi} 
Let $X\rightarrow Y$ be an abelian geometric cover as above with $v_1\in S$ a ramified prime which is totally ramified in $X\rightarrow Y'$. 
Using the decomposition $\widetilde{G}=G\times H$, for any prime $\pr$ write $\Fr_{\pr}=(\Fr_{\pr,G},\Fr_{\pr,H})\in G\times H$.
If $\chi(\pi_H(I_v))\neq 1$, then $H_v(\chi)=0$, otherwise
\[ H_v(\chi)=W[G][[G_\F]]/(1-\chi(\Fr_{v,H}^{-1})\Fr_{v,G}^{-1}\gamma^{-d_v},\pi_G(\tau)-1 : \tau\in I_v) .\]
Moreover, we have an exact sequence
\[ H_v(\chi)\iri W[G][[G_\F]]/(1-\chi(\Fr_{v,H}^{-1})\Fr_{v,G}^{-1}\g^{-d_v})\sri
W[G][[G_\F]]/(1-\chi(\Fr_{v,H}^{-1})\Fr_{v,G}^{-1}\g^{-d_v},n(\pi_G(I_v))). \]
\end{lem}

\begin{proof}
Recall that
\[ H_v=\Z_p[\widetilde{G}][[G_\F]]/(1-\Fr_v\gamma^{-d_v},\tau-1:\tau\in I_v)=\Z_p[\widetilde{G}/I_v][[G_\F]]/(1-\Fr_v\gamma^{-d_v}).\] 
If $\chi(\pi_H(I_v))\neq 1$, then we immediately have $H_v(\chi)=0$.

Now assume $\chi(\pi_H(I_v))=1$. Then $e_{\chi}(\tau-1)=(\pi_G(\tau)-1)e_{\chi}\in W[G]$ for any $\tau\in I_v$, where 
$\pi_G(\tau)-1\in\Z[\pi_G(I_v)]$. Thus
\[ e_{\chi}((1-\Fr_v\gamma^{-d_v}))=(1-\chi(\Fr_{v,H})\Fr_{v,G}\gamma^{-d_v})e_{\chi}\in W[G][G_{\F}].\]
Therefore, taking $\chi$-parts in the exact sequence 
\[ 0\rightarrow(1-\Fr_v\gamma^{-d_v},\tau-1:\tau\in I_v)\rightarrow \Z_p[\widetilde{G}][[G_\F]]\rightarrow H_v\rightarrow 0,\]
we have that
\[ H_v(\chi)=W[G][[G_\F]]/(1-\chi(\Fr_{v,H}^{-1})\Fr_{v,G}^{-1}\gamma^{-d_v},\pi_G(\tau)-1 : \tau\in I_v ). \]
The exact sequence follows as in \eqref{eqap2} using the remarks after \cite[Lemma 2.4]{GP2}: indeed we have a map 
$H_v(\chi)\rightarrow W[G][G_\F]/(1-\chi(Fr_{v,H}^{-1})F_{v,G}^{-1}\gamma^{-d_v})$, sending ${\bf 1}_{H_v}$ to $n(\pi_G(I_v))$, 
and all is clear in the sequence except the injectivity of this map. To work with $W$-modules we observe that
$1-\chi(Fr_{v,H}^{-1})F_{v,G}^{-1}\gamma^{-d_v}$ is a polynomial of degree $d_v$ in $\gamma^{-1}$ with leading term a unit in $W[G]$,
therefore $W[G][G_{\F}]/(1-\chi(Fr_{v,H}^{-1})F_{v,G}^{-1}\gamma^{-d_v})$ is $W$-free of rank $d_v|G|$. Similarly $H_v(\chi)$ is 
$W$-free of rank $d_v|G/\pi_G(I_v)|$ and exactness follows.
\end{proof}

\begin{rem}\label{RemNotRamInG}
Note that whenever $v$ is only ramified in $Y'\rightarrow Y$, i.e. $\pi_G(I_v)=0$, we have $n(\pi_G(I_v))=1$ and the exact sequence 
mentioned in the previous lemma reduces to the isomorphism
\[ H_v(\chi)\simeq W[G][[G_\F]]/(1-\chi(\Fr_{v,H}^{-1})\Fr_{v,G}^{-1}\gamma^{-d_v}).\]
\end{rem}

The decomposition of modules in $\chi$-parts obviously reflects on their Fitting ideals as well, hence we define
\begin{defin}\label{DefSticFuncCyc} For any $\chi\in \Hom(H,W^*)$, the {\em $\chi$-Stickelberger series} for $K/F$ is
\[ \Theta_{K/F,S,\chi}(u):=\chi(\Theta_{K/F,S})(u)\in W[G][[u]] \,.\]
\end{defin}

\noindent The usual fundamental relations for idempotents yield
\begin{equation}\label{EqChiTheta} 
e_\chi\Theta_{K/F,S}(u)=\Theta_{K/F,S,\chi}(u)e_{\chi} 
\end{equation}
and
\[ \Theta_{K/F,S}(u)=\sum_\chi \Theta_{K/F,S,\chi}(u) e_\chi\,. \]

Following the proof of Theorem \ref{teoap2} and using Lemma \ref{lem2.4modchi} we obtain 

\begin{thm}\label{teoap2b} 
Assume $X\rightarrow Y$ is an abelian geometric cover as above. If $\chi\neq\chi_0$, then
\begin{equation}\label{AppFitt2}
\begin{array}{ll} \Fitt_{W[G][[G_\F]]}(T_p(Jac(X)(\ov{\F}))(\chi)^*) & = 
\displaystyle{  \frac{\Theta_{K/F,S,\chi}(\g^{-1})}{\displaystyle{\prod_{v\in S_\chi^{(1)}} } 
(1-\chi(\Fr_{v,H})^{-1}\Fr_{v,G}^{-1}\g^{-d_v})} } \\
\ & \cdot \displaystyle{ \prod_{ v\in S_{\chi}-S_\chi^{(1)}}
\left(1,\frac{n(\pi_G(I_v))}{ 1-\chi(\Fr_{v,H})^{-1}\Fr_{v,G}^{-1}\g^{-d_v} }\right)}\,,\end{array}
\end{equation}
where $^*$ now denotes the $W$-dual, $S_{\chi}:=\{v\in S : \chi(\pi_H(I_v))=1\}$ and
$S_\chi^{(1)}:=\{v\in S_\chi : \pi_G(I_v)=0 \}$.\\
If $\chi=\chi_0$, then
\begin{equation}\label{AppFitt3}
\begin{array}{ll} \Fitt_{W[G][[G_\F]]}(T_p(Jac(X)(\ov{\F}))(\chi_0)^*) & = 
\displaystyle{  \frac{\Theta_{K/F,S,\chi_0}(\g^{-1})}{\displaystyle{ \prod_{\begin{subarray}{c}v\in S_{\chi_0}^{(1)}\\ v\neq v_1\end{subarray}}
(1-\Fr_{v,G}^{-1}\gamma^{-d_v})} }  }\\
\ & \displaystyle{ \cdot \left(1,\frac{n(G)}{\frac{1-\g^{-d_{v_1}}}{1-\g^{-1}}}\right) \cdot
\prod_{\begin{subarray}{c} v\in S-S_{\chi_0}^{(1)} \\v\neq v_1\end{subarray}} 
\left(1,\frac{n(\pi_G(I_v))}{(1-\Fr_{v,G}^{-1}\g^{-d_v})}\right)\,.}\end{array}
\end{equation}
\end{thm}

\begin{proof}
Assume first $\chi\neq\chi_0$. From Lemmas \ref{lem2.3gen} and \ref{lem2.4modchi} (and recalling Remark \ref{RemNotRamInG})
we have the four term exact sequence
\[ T_p(Jac(X)(\ov{\F}))(\chi) \iri T_p(\calm_{\ov{S}})(\chi) \rightarrow
\bigoplus_{v\in S_{\chi}}W[G][[G_\F]]/(1-\chi(\Fr_{v,H}^{-1})\Fr_{v,G}^{-1}\g^{-d_v})\] \vspace{-1truecm}
\[ \hspace{4.5truecm} \xymatrix{ \ar@{->>}[d] \\ \ } \]
\vspace{-.3truecm}
\[ \hspace{3truecm} \bigoplus_{v\in S_{\chi}-S_\chi^{(1)}}W[G][[G_\F]]/((1-\chi(\Fr_{v,H}^{-1})\Fr_{v,G}^{-1}\g^{-d_v}),n(\pi_G(I_v))) .\] 
Denote by $X_2$, $X_3$ and $X_4$ the second, third and fourth modules appearing in the sequence above. All modules are finitely
generated and free over $W$, moreover $X_3$ has projective dimension 0 or 1 over $W[G][[G_\F]]$, while $X_2$ has projective dimension 1
over $W[G][[G_\F]]$ because it has no non-trivial finite submodules and is $G$-cohomologically trivial by (an easy application of) 
Theorem \ref{teoap1}. Then \cite[Lemma 2.4]{GP2} yields
\[ \Fitt_{W[G][[G_\F]]}(T_p(Jac(X)(\ov{\F}))(\chi)^*)\Fitt_{W[G][[G_\F]]}(X_3) \!\! = \!\!
\Fitt_{W[G][[G_\F]]}(X_2)\Fitt_{W[G][[G_\F]]}(X_4).\] 
Since
\[ \Fitt_{W[G][[G_\F]]}(X_3)=\left(\prod_{v\in S_{\chi}}(1-\chi(\Fr_{v,H}^{-1})\Fr_{v,G}^{-1}\g^{-d_v})\right) \]
and
\[ \Fitt_{\Z_p[G][[G_\F]]}(X_4)=\prod_{v\in S_{\chi}-S_\chi^{(1)}}(1-\chi(\Fr_{v,H}^{-1})\Fr_{v,G}^{-1}\g^{-d_v},n(\pi_G(I_v))) \ ,\]
Theorem \ref{teoap1} immediately implies equation \eqref{AppFitt2}.

Assume now $\chi=\chi_0$: obviously $S_{\chi_0}=S$. Since $H_{v_1}(\chi_0)=W[[G_{\F}]]/(1-\gamma^{-d_{v_1}})$, with
the same argument of the proof of Theorem \ref{teoap2}, we find an analog of the exact sequence (\ref{eqap3}) that now reads as
\[  (\g-1)H_{v_1}(\chi_0) \iri W[G][[G_\F]]/\left(\frac{1-\g^{-d_{v_1}}}{1-\g^{-1}}\right) \sri
W[G][[G_\F]]/\left(\frac{1-\g^{-d_{v_1}}}{1-\g^{-1}},n(G)\right) .\]
As above we obtain the following four term exact sequence 
\[ T_p(Jac(X)(\ov{\F}))(\chi_0) \iri T_p(\calm_{\ov{S}})(\chi_0) \rightarrow
W[G][[G_\F]]/\left(\frac{1-\g^{-d_{v_1}}}{1-\g^{-1}}\right) \bigoplus_{\begin{subarray}{c} v\in S\\ v\neq v_1 \end{subarray}}
\displaystyle{ \frac{W[G][[G_\F]]}{(1-\Fr_{v,G}^{-1}\g^{-d_v})} } \]
\vspace{-1truecm}
\[ \hspace{4.5truecm} \xymatrix{ \ar@{->>}[d] \\ \ } \]
\vspace{-.5truecm}
\[ \hspace{2truecm} W[G][[G_\F]]/\left(\frac{1-\g^{-d_{v_1}}}{1-\g^{-1}},n(G)\right) 
\bigoplus_{\begin{subarray}{c} v\in S-S_{\chi_0}^{(1)}\\ v\neq v_1 \end{subarray}}W[G][[G_\F]]/(1-\Fr_{v,G}^{-1}\g^{-d_v},n(\pi_G(I_v)))  \] 
(because of $\chi_0$ we could actually use $\Z_p$ in place of $W$ here, we keep $W$ for coherence with the other formulas). 
Denote by $Y_3$ and $Y_4$ the third and fourth modules appearing in this last sequence.
Another application of \cite[Lemma 2.4]{GP2}, the facts that
\[ \Fitt_{W[G][[G_\F]]}(Y_3)=\left(\frac{1-\g^{-d_{v_1}}}{1-\g^{-1}}\right)\cdot
\left(\prod_{\begin{subarray}{c} v\in S\\ v\neq v_1 \end{subarray}}(1-\Fr_{v,G}^{-1}\g^{-d_v})\right) \]
and
\[ \Fitt_{\Z_p[G][[G_\F]]}(Y_4)=\left(\frac{1-\g^{-d_{v_1}}}{1-\g^{-1}},n(G)\right)\cdot
\prod_{\begin{subarray}{c} v\in S-S_{\chi_0}^{(1)} \\ v\neq v_1 \end{subarray}}(1-\Fr_{v,G}^{-1}\g^{-d_v},n(\pi_G(I_v))) \ ,\]
and Theorem \ref{teoap1} immediately imply equation \eqref{AppFitt3}.
\end{proof}

\begin{rem}\label{RemExtra0}
The first factor appearing in the Fitting ideals of Theorem \ref{teoap2b} hints at the possibility of finding {\em extra zeroes} for the 
Stickelberger element when specializing $\gamma$ to 1. As predictable the number of such zeroes depends on the behavior of 
ramified primes and, in particular, on the cardinality of the set $S_\chi^{(1)}$. We shall see an explicit example in the next section.
\end{rem}

Similarly we can obtain an analog of Theorem \ref{teoap4} but we can safely leave this to the reader (there is nothing new in that computation
but formulas get a bit involved), and move to the arithmetic applications
for the Carlitz-Hayes cyclotomic extensions.

\end{section}

\begin{section}{Fitting ideals in Carlitz-Hayes cyclotomic extension}\label{SecFittIdCycExt}
We consider a well-known abelian cover arising from class field theory.
Let $\Phi$ be a sign normalized rank 1 Drinfeld module, i.e. a ring homomorphism $\Phi:A\rightarrow H_A\{\tau\}$, where $H_A\{\tau\}$ is
the ring of skew-polynomials in the variable $\tau$ with coefficients in $H_A$ (for more details see \cite[Chapter 7]{Goss}). For any
ideal $\mathfrak{a}$ of $A$, denote by $\Phi[\mathfrak{a}]$ the $\mathfrak{a}$-torsion of $\Phi$: we recall that $H_A\{\tau\}$ is
right-euclidean so, if we let $\Phi_{\mathfrak{a}}$ be the unique monic generator of the ideal generated by $(\Phi(a):a\in\mathfrak{a})$,
we have 
\[ \Phi[\mathfrak{a}]:=\{ x\in F^{sep} : \Phi_{\mathfrak{a}}(x)=0\} \,,\]
where $F^{sep}$ is a separable closure of $F$ and we interpret $\Phi_{\mathfrak{a}}$ as a polynomial via $\tau(x):=x^{\#\F}$.

It is well-known that $F(\mathfrak{a}):=H_A(\Phi[\mathfrak{a}])$ is an abelian extension of $F$ where the only ramified primes are the
ones dividing $\mathfrak{a}$ and $\infty$. The inertia group of $\infty$ coincides with its decomposition group and is isomorphic to
$\F^*$ (recall that we assume $\deg(\infty)=d_\infty=1$). Moreover, when $\mathfrak{a}=\pr^n$ is a power of a prime $\pr$, 
$F(\mathfrak{a})/H_A$ is totally ramified at the place $\pr$ (see \cite[Chapter 7]{Goss}). The infinite Galois extension 
$\calf_\pr:=\cup_n F(\pr^n)$ has properties similar to the ones of $\cup_n k(\mu_{p^n})$ for a number field $k$, hence
it will be called the {\em $\pr$-cyclotomic extension} of $F$. To apply the results of the previous section we consider
\[ F_0:=F=\F(Y)\qquad F(\pr):=F_1=\F(Y')\qquad F_n:=F(\pr^n)=\F(X_n)\ n>1 ,\]
we use $X_n$ in place of $X$ (and later on $G_n$, $I_{\infty,n}$, $I_{\pr,n}$ and so on) to keep track of the {\em layer} of the cyclotomic extension 
we are dealing with. To fit the previous hypotheses, we assume from now on that $h^0(F)=[H_A:F]$ is prime with $p$ so that 
$[F_1:F]=(q^{d_{\pr}}-1)h^0(F)$ is prime with $p$ as well. We put $W$ as a ring extension of $\Z_p$ containing $\mu_{|Gal(F_1/F)|}$ and recall 
that $G_n$ is a $p$-group. The rest of the setting of Section \ref{SecFittId} translates into
\[ Gal(F_n/F)=\widetilde{G}_n\simeq Gal(F_n/F_1)\times Gal(F_1/F)=:G_n\times H.\]
 The set of ramified primes is $S=\{\pr,\infty\}$ with $I_\infty:=I_{\infty,n}\simeq \F^* \iri H$ and $F_n^{I_{\pr,n}}=H_A$ for any $n$. 

By Lemma \ref{LemGP} we have that
\[H_{n,\infty}\simeq \Z_p[G_n\times H/\F^*][[G_{\F}]]/(1-\gamma^{-1})\]
(note that we can choose $\Fr_{\infty}=1$ because the decomposition and inertia groups of $\infty$ coincide) and
\[ H_{n,\pr}\simeq\Z_p[Gal(H_A/F)][[G_{\F}]]/(1-\Fr_{\pr}\gamma^{-d_{\pr}}).\]

 \begin{defin} \label{DefCharType}
Let $\chi$ be a character of $H$, i.e. $\chi\in Hom(H,W^*)$. We distinguish three types of characters
\begin{itemize}
\item $\chi$ is said to be of type 1 if $\chi(I_{\infty})\neq 1$;
\item $\chi$ is said to be of type 2 if $\chi(I_{\infty})=1$ and $\chi(Gal(F_1/H_A))\neq 1$;
\item $\chi$ is said to be of type 3 if $\chi(Gal(F_1/H_A))=1$.
\end{itemize}
Among the characters of type 3, there is the trivial one which will be denoted as usual by $\chi_0$.
 \end{defin}

From Lemma \ref{lem2.4modchi} we immediately obtain
\[ H_{n,\pr}(\chi)\simeq\left\{\begin{array}{cl}
0 & {\rm if}\ \chi\ {\rm is\ of\ type}\ 1\ {\rm or}\ 2\\
W[[G_{\F}]]/(1-\chi(\Fr_\pr^{-1})\gamma^{-d_\pr}) & {\rm otherwise}
\end{array}\right. \]
and 
\[ H_{n,\infty}(\chi)\simeq\left\{\begin{array}{cl}
0 & {\rm if}\ \chi\ {\rm is\ of\ type}\ 1\\
W[G_n] & {\rm otherwise.}
\end{array}\right. \]

Characters of type 1 and 2 have a behavior similar to nontrivial characters for the basic case of the rational function field $\F(t)$
(detailed in \cite{ABBL}): indeed for those characters we can extend the results of the previous section to include the computation 
of the Fitting ideals of class groups (not only of its dual). We denote by $T_p(F_n)$ the Tate module of the $\ov{\F}$-rational
points of the Jacobian of $X_n$ (which here plays the role of $T_p(Jac(X)(\ov{\F}))$ of the previous section). 

\begin{thm}\label{teo3.2} 
Let $\chi$ be a character of type 1 or 2. We have
\[ \Fitt_{W[G_n][[G_\F]]}(T_p(F_n)(\chi))=\left\{\begin{array}{cl}
\left(\Theta_{F_n/F,S,\chi}(\gamma^{-1})\right) & if\ \chi\ is\ of\ type\ 1\\
\displaystyle{\left(\frac{\Theta_{F_n/F,S,\chi}(\gamma^{-1})}{1-\gamma^{-1}}\right)} & if\ \chi\ is\ of\ type\ 2 .
\end{array}\right. \] 
\end{thm}

\begin{proof}
The exact sequence \eqref{EqChiPart} here reads as
\begin{equation} \label{ExSeqType12}
 0\rightarrow T_p(F_n)(\chi)\rightarrow T_p(\mathcal{M}_{n,S_n})(\chi)\rightarrow L_n(\chi)\rightarrow 0,
\end{equation}
(where $S_n$ is the set of primes of $F_n$ lying above primes in $S$) and Theorem \ref{teoap1} yields (taking $\chi$-parts)
\[ Fitt_{W[G_n][[G_{\F}]]}(T_p(\mathcal{M}_{n,S_n})(\chi))=(\Theta_{F_n/F,S,\chi}(\gamma^{-1})). \]
By Lemma \ref{lem2.3gen} and the computations above we have
\[ L_n(\chi)=\left\{\begin{array}{cl} 0 & {\rm if}\ \chi\ {\rm is\ of\ type}\ 1\\
W[G_n] & {\rm if}\ \chi\ {\rm is\ of\ type}\ 2.
\end{array}\right.\] 
Now the statement for type 1 characters is obvious. For characters of type 2 apply \cite[Lemma 3]{CG} to the sequence
\eqref{ExSeqType12}  to obtain
\[ (1-\gamma^{-1})\Fitt_{W[G_n][[G_{\F}]]}(T_p(F_n)(\chi))=\Fitt_{W[G_n][[G_{\F}]]}(T_p(\mathcal{M}_{n,S_n})(\chi))\]
and conclude the proof.
\end{proof}

For all characters, as a consequence of Theorem \ref{teoap2b}, we have
\begin{thm}\label{teo3.3} 
With notations and hypotheses as above, $\Fitt_{W[G_n][[G_\F]]}(T_p(F_n)(\chi)^*)$ is equal to:
\[ \begin{array}{ll}
\left(\Theta_{F_n/F,S,\chi}(\gamma^{-1})\right) & \chi\ of\ type\ 1 \\
\displaystyle{ \left( \frac{\Theta_{F_n/F,S,\chi}(\gamma^{-1})}{1-\gamma^{-1}} \right) } & \chi\ of\ type\ 2 \\
\displaystyle{ \frac{\Theta_{F_n/F,S,\chi}(\gamma^{-1})}{(1-\gamma^{-1})}\cdot\left(1,\frac{n(G_n)}{1-\chi(\Fr_{\pr,H})^{-1}\gamma^{-d_{\pr}}}\right) }& 
\chi\ of\ type\ 3,\ \chi(\Fr_{\pr,H})\neq 1 \\
\displaystyle{ \frac{\Theta_{F_n/F,S,\chi}(\gamma^{-1})}{(1-\gamma^{-1})^2}\cdot
\left(1-\gamma^{-1},\frac{n(G_n)}{1+\gamma^{-1}+\ldots+\gamma^{-d_{\pr}+1}} \right) } &
\chi\neq \chi_0\ of\ type\ 3,\ \chi(\Fr_{\pr,H})=1\\
\displaystyle{ \frac{\Theta_{F_n/F,S,\chi}(\gamma^{-1})}{1-\gamma^{-1}}\cdot\left(1,\frac{n(G_n)}{1+\gamma^{-1}+\ldots+\gamma^{-d_{\pr}+1}}\right) }
& \chi=\chi_0
\end{array} . \]
\end{thm}

\begin{proof} Just specialize the formulas of Theorem \ref{teoap2b} to the various cases. Note that $S=\{\pr,\infty\}$ so the 
$S_\chi=\{v \in S: \chi(\pi_H(I_v))=1\}$ of that theorem is empty if $\chi$ is of type 1, it contains only $\infty$ if $\chi$ is of type 2 
 and is equal to $S$ if $\chi$ is of type 3. Moreover since $I_\infty \iri H$ 
we have $S_\chi^{(1)}=\{\infty\}$ whenever $\chi$ is not of type 1. Finally $\pr$ plays the role of $v_1$, we already mentioned that
$\Fr_\infty=1$ and, since $\pr$ is totally ramified in $F_n/F_1$, we have $\Fr_{\pr,G_n}=1$ as well.
\end{proof}

\begin{rem}
For characters of type 1 or 2  the Fitting ideals of the Tate module and of its $W^*$ dual coincide (those are the only cases in which 
we can compute both). 
\end{rem}

\subsection{Fitting ideals of class groups}\label{SecFittClassGroups}
Let $C_n:=\calCl(F_n)\{p\}$ be the $p$-torsion of the class groups of degree zero divisors of $F_n$, which is a 
$\mathbb{Z}_p[G_n\times H]$-module.

It is well-known that $C_n$ can be obtained as the $G_\F$-coinvariants of the Tate module $T_p(F_n)$ (see, for example,
\cite[Lemma 4.6]{ABBL} or \cite[Lemma 2.4.1]{CPhD}), therefore thanks to the properties of Fitting ideals we can compute the Fitting
ideals of class groups (or of their duals) simply by specializing the previous formulas with $\gamma\mapsto 1$.

For the dual of the class groups as noted in \cite[Section3]{GP2} (see also \cite[Remark 4.7]{ABBL}), we know that 
\[ C_n^{\vee}:=Hom(C_n,\Q_p/\Z_p)\simeq T_p(F_n)^*/(1-\gamma^{-1})T_p(F_n)^*\] 
therefore we have
\[ \Fitt_{W[G_n]}(C_n(\chi)^{\vee})=\pi^{W[G_n][[G_\F]]}_{W[G_n]}(\Fitt_{W[G_n][[G_\F]]}(T_p(F_n)(\chi)^*)) .\]
We do not write down the class group case because we only have formulas for the characters of type 1 and 2 and they coincide with
the ones for the dual.

\begin{cor}\label{CorDualClGr}
With notations and hypotheses as above, $ \Fitt_{W[G_n]}(C_n(\chi)^{\vee})$ is equal to 
\[ \begin{array}{ll}
\left( \Theta_{F_n/F,\chi,S}(1)\right) & \chi\ of\ type\ 1 \\
\displaystyle{ \left( \frac{\Theta_{F_n/F,S,\chi}(\gamma^{-1})}{1-\gamma^{-1}}_{|\gamma=1} \right) } & \chi\ of\ type\ 2 \\
\displaystyle{ \frac{\Theta_{F_n/F,S,\chi}(\gamma^{-1})}{(1-\gamma^{-1})}_{|\gamma=1}\cdot\left(1,\frac{n(G_n)}{1-\chi(\Fr_{\pr,H})^{-1}}\right) } & 
\chi\ of\ type\ 3,\ \chi(\Fr_{\pr,H})\neq 1 \\
\displaystyle{ \frac{\Theta_{F_n/F,S,\chi}(\gamma^{-1})}{(1-\gamma^{-1})^2}_{|\gamma=1}\cdot
\left(\frac{n(G_n)}{d_\pr} \right) } & \chi\neq \chi_0\ of\ type\ 3,\ \chi(\Fr_{\pr,H})=1\\
\displaystyle{ \frac{\Theta_{F_n/F,S,\chi}(\gamma^{-1})}{1-\gamma^{-1}}_{|\gamma=1}\cdot\left(1,\frac{n(G_n)}{d_\pr}\right) }
& \chi=\chi_0
\end{array}\]
\end{cor}

\subsection{Limits of Fitting ideals}\label{SecLimFitt}
For characters of type 1 and 2 (two of the three cases in which we have a principal Fitting ideal) it is possible to consider
the inverse limit of the $C_n$ with respect to the maps induced by the natural norm maps $N_{F_m/F_n} : C_m \rightarrow C_n$
(for any $m\geqslant n$) and compute the Fitting ideal of that limit as the inverse limit of the Fitting ideals of Corollary \ref{CorDualClGr}
(without taking duals of course). This provides an element in the Iwasawa algebra $\Lambda:=W[[Gal(\calf_\pr/F_1)]]$ which can be used as the
algebraic counterpart of a $\pr$-adic $L$-function in the Iwasawa Main Conjecture for this setting. Details for the case of the Carlitz module
(i.e. $F=\F(t)$\,) are in \cite{ABBL}, the generalization to the case presented here can be found in the third author PhD thesis \cite{CPhD}. 

Here we would like to deal briefly with limits for Fitting ideals computed for characters of type 3: a motivation comes from the usual
application to Iwasawa theory, unfortunately, since we can only work with duals, it is not immediately clear what kind of arithmetic
meaning can be  associated with these limits (i.e. which is the Iwasawa module they are related to).

Directly from the definition of Stickelberger series we have 
\[ \pi^{n+m}_n(\Theta_{F_{n+m}/F,S,\chi}(u))=\Theta_{F_n/F,S,\chi}(u) .\]
So they are compatible with respect to the projection maps defining the Iwasawa algebra $\Lambda$ as an inverse limit of the 
group rings $W[G_n]$. Thus we can define
\begin{equation}\label{EqStickSer} 
\Theta_{\calf_\pr/F,S,\chi}(u):=\plim{n} \Theta_{F_n/F,S,\chi}(u) \in \Lambda[[u]] ,
\end{equation}
and its specialization
\begin{equation}\label{EqStickEl} 
\Theta_{\calf_\pr/F,S,\chi}(\gamma^{-1}):=\plim{n} \Theta_{F_n/F,S,\chi}(\gamma^{-1}) \in \Lambda[[G_\F]] .
\end{equation}

Regarding the other element appearing as a generator of the Fitting ideals we have
\[ \pi^{n+m}_n(n(G_{n+m}))=[F_{n+m}:F_n]n(G_n) .\]
Hence for any $m\geqslant 0$
\[ \pi^{n+m}_n(\Fitt_{W[G_{n+m}]}(C_{n+m}(\chi)^{\vee})) \subseteq \Fitt_{W[G_n]}(C_n(\chi)^{\vee}) \]
for any $\chi$ of type 3 (note that we actually have equality for characters of type 1 and 2) and this is enough to have a compatible
inverse system of ideals and to define their inverse limit inside $\Lambda[[G_\F]]$ (a similar procedure has been used in \cite{BBL1} and 
\cite{BBL2} for characteristic ideals).

\begin{defin}\label{DefProFitt}
With the above notations we define the {\em pro-Fitting ideal} of the dual of the $\chi$-part of the class groups as
\[ \widetilde{\Fitt}_\Lambda(\calc^\vee(\chi)):=\plim{n} \Fitt_{W[G_n]}(C_n(\chi)^{\vee}).\]
\end{defin}

\begin{rem}\label{RemLimClGroups}
For characters of type 1 and 2 we have 
\[ \widetilde{\Fitt}_\Lambda(\calc^\vee(\chi)) = \plim{n} \Fitt_{W[G_n]}(C_n(\chi)) \]
as well. Moreover, thanks to the presence of a totally ramified prime, one can also show that 
\[ \plim{n} \Fitt_{W[G_n]}(C_n(\chi)) = \Fitt_\Lambda(\plim{n} C_n(\chi)) \]
(see \cite[Sections 4.4 and 5]{ABBL} or \cite[Sections 2.3 and 2.4]{CPhD}). As mentioned above, this provides a link between the
Fitting ideal of the inverse limit of the class groups of subextensions of $\calf_\pr$ (i.e. the {\em class group} of $\calf_\pr$)
and a Stickelberger element. Such formula is one of the incarnations of Iwasawa Main Conjecture for the function field setting,
its relation with more analytic objects (like Goss $\zeta$-function) will be briefly explained in the next section.
\end{rem}

\begin{cor}\label{CorProFitt}
For characters of type 3 one has
\[ \widetilde{\Fitt}_\Lambda(\calc^\vee(\chi)) = \left\{ \begin{array}{ll} 
\displaystyle{ \left( \frac{\Theta_{\calf_\pr/F,S,\chi}(\gamma^{-1})}{(1-\gamma^{-1})}_{|\gamma=1} \right)}& 
if\ \chi(\Fr_{\pr,H})\neq 1 \\
(0) & if\ \chi\neq \chi_0,\ \chi(\Fr_{\pr,H})=1\\
\displaystyle{ \left( \frac{\Theta_{\calf_\pr/F,S,\chi}(\gamma^{-1})}{1-\gamma^{-1}}_{|\gamma=1} \right)} & if\ \chi=\chi_0
\end{array} \right. .\]
\end{cor}

\begin{proof}
Just use limits on the ideals of Corollary \ref{CorDualClGr} and note that $\plim{n} n(G_n)=0$.
\end{proof}

\begin{rem}\label{RemLimOnCores}
Via the maps
\[ \lambda_n : W[G_n]^\vee \rightarrow W[G_n] \]
\[ \lambda_n(\varphi):=\sum_{g\in G_n} \varphi(g) \]
one can obtain a self duality for Iwasawa rings (and for their limit, i.e. for the Iwasawa algebra). Moreover
such maps provide commutative diagrams with projections and corestrictions (see, e.g., \cite[Appendix A]{Kid}).
It might be interesting to study direct limits as well for our Fitting ideals with respect to corestriction maps.
We can immediately remark that $cor^n_{n+1}(n(G_n))=n(G_{n+1})$, the corestriction map on Stickelberger elements is natural and 
already appeared in Theorem \ref{teoap4} and, between duals of class groups, the natural maps induced by norms go
in the direction of a direct limit (i.e.  $(N^{n+1}_n)^\vee: C_n(\chi)^\vee \rightarrow C_{n+1}(\chi)^\vee$).
One has to check compatibility of all maps involved and figure out the behavior of Fitting ideals with respect to such a limit,
i.e. does computation of Fitting ideals commute with the direct limit in this setting (as it does with inverse limit as mentioned in 
Remark \ref{RemLimClGroups}) ? The main reason for not dealing with these issues here is that it is not clear what kind of arithmetic 
information (if any) one can obtain from such a procedure, but the appearance of Stickelberger elements (with their several relations
with $L$-functions in general) could be a motivation for investigating these direct limits in some future work.
\end{rem}

\end{section}

\begin{section}{Interpolation of Goss $\zeta$-function via Stckelberger elements}\label{SecGoss}
In this final section we would like to briefly explain how Stickelberger element (and Stickelberger series in general) can be used to interpolate
the key analytic object for the function field setting, i.e. the {\em Goss $\zeta$-function} whose central role in the theory is just one among the many
key contributions David Goss provided to the subject.

Here we still deal with the cyclotomic extension $\calf_\pr:=\cup_{n\geqslant 1} F_n$.
We start by recalling the definition of the Goss $\zeta$-function in our (slightly simplified) setting where we assume that $p$ does not divide
the order of the class group of $F$ (details for a general function field are provided in \cite[Chapter 8]{Goss}). 
Let $F_\infty$ denote the completion of $F$ at the  prime at infinity and let $\C_{\infty}$ be the completion of 
a fixed algebraic closure of $F_\infty$. We fix a sign function $\text{sgn}: F_\infty^*\rightarrow\F^*$ and a positive uniformizer 
$\pi_\infty\in F_\infty$, i.e. with $\text{sgn}(\pi_\infty) = 1$. The 1-units of $F_\infty$ will be denoted by $U^1_\infty$.
Since the sign function has image in $\F^*$ (which has order prime with $p$), our hypothesis on the class group implies that
the group $\mathcal{I}_F/\mathcal{P}^+$ (integral ideals of $F$ modulo principal ideals generated by a positive element) has order prime with
$p$ as well. Let $h:=|\mathcal{I}_F/\mathcal{P}^+|$, then, for any ideal $\mathfrak{a}$ of $A$, one has $\mathfrak{a}^h=(\alpha)$ for some 
positive $\alpha\in A$. Denote by $\langle \alpha \rangle_\infty\in U^1_\infty$ the 1-unit associated to $\alpha$ and define
\[ \langle \mathfrak{a}\rangle :=\langle \alpha\rangle_\infty^{1/h} \in U^1_\infty \]
as the {\em unique} 1-unit whose $h$-power is $\langle\alpha\rangle_\infty$ (it is still in $U^1_\infty$ because of our hypothesis on $h$ 
and Hensel's Lemma). We drop the index $\infty$ because in this section we only deal with interpolation at the prime at infinity, for interpolation
at different primes see \cite[Section 3.3]{ABBL} and \cite[Section 1.8]{CPhD}. Note that the hypothesis $d_\infty=1$ makes this definition
independent from the choice of the uniformizer $\pi_\infty$ (and, obviously, also from the choice of one of its $d_\infty$-th roots).

\begin{defin}\label{DefGossZeta}
For $s=(x,y)\in\bbS_\infty:=\C_{\infty}^*\times\Z_p$ define the exponential of an ideal $\mathfrak{a}$ by
\[ \mathfrak{a}^s := x^{\text{deg}\,\mathfrak{a}}\langle\mathfrak{a}\rangle^y .\]
The {\em Goss $\zeta$-function} for $F$ is
\[ \zeta_A(s) := \sum_{\mathfrak{a}\neq 0} \mathfrak{a}^{-s}\,,\]
where the sum runs through the set of all the non-zero integral ideals of $A$.
\end{defin}

\noindent One of the more relevant properties of Goss $\zeta$-function (which trivially converges for $|x|_\infty > 1$) is that it can be extended analytically 
to the whole space $\bbS_\infty$ by using some inequalities arising from the Riemann-Roch Theorem (see \cite[Section 8.9]{Goss}).

Let $G_S:=Gal(F_S/F)$ be the Galois group of the maximal abelian extension of $F$ unramified outside $S$, then $Gal(\calf_\pr/F)$ is
a quotient of $G_S$. Let $\mathcal{W}_S$ be the subgroup of $G_S$ generated by all the Frobenius $\Fr_v$ ($v\not\in S$) and let
$M_S$ be the fixed field of the topological closure of $\mathcal{W}_S$. Since all primes not in $S$ split completely in $M_S$,
Chebotarev density theorem yields $M_S=F$ and $G_S=\ov{M_S}$. Therefore to define a continuous character on $G_S$ (and, a fortiori, on 
$Gal(\calf_\pr/F)$\,) it suffices to describe its values on $\Fr_v$ for all $v\not\in S$.  

For any $y\in \Z_p$, consider the $\C^*_\infty$-valued character 
\[  \Psi_y : Gal(\calf_\pr/F) \rightarrow \C^*_\infty \]
\[ \Psi_y(\Fr_{\mathfrak{q}}):=\langle \mathfrak{q} \rangle^{-y} \]
and extend it in a natural way to a map $\Lambda[[u]] \rightarrow \C_\infty[[u]]$ (still denoted by $\Psi_y$ by a little abuse of notations).

\begin{thm}\label{ThmStickGoss}
For any $s=(x,y)\in \bbS_\infty$,
\[ \Psi_y(\Theta_{\calf_\pr/F,S})(x)=(1-\pr^s)\zeta_A(-s)\,.\]
\end{thm}

\begin{proof}
Directly from Definition \ref{DefGossZeta} one can see that there exists an Euler product formula for $\zeta_A$ which reads as
\begin{equation}\label{EqEuProdZeta}
\zeta_A(s)=\prod_{\mathfrak{q}\neq \infty} (1-\mathfrak{q}^{-s})^{-1}=
\prod_{\mathfrak{q}\neq \infty} \left(1-\langle \mathfrak{q}\rangle^{-y}x^{-\deg(\mathfrak{q})}\right)^{-1} 
\end{equation}
(because $-s=(x^{-1},-y)$\,).

\noindent Now from Definition \ref{DefStick} and equation \eqref{EqStickSer} we have 
\begin{equation}\label{EqCharStick}
\begin{array}{ll} \Psi_y(\Theta_{\calf_\pr/F,S})(u) & =\displaystyle{ \prod_{\mathfrak{q}\not\in S} 
\left( 1-\Psi_y(\Fr_\mathfrak{q}^{-1})u^{\deg(\mathfrak{q})}\right)^{-1} }\\
\ & = \displaystyle{ \prod_{\mathfrak{q}\not\in S} \left( 1-\langle \mathfrak{q}\rangle^y u^{\deg(\mathfrak{q})}\right)^{-1}} .
\end{array} 
\end{equation}
Specializing $u\mapsto x$ (and comparing with \eqref{EqEuProdZeta}) we obtain 
\[ \begin{array}{ll} \Psi_y(\Theta_{\calf_\pr/F,S})(x) & = 
\displaystyle{ \prod_{\mathfrak{q}\not\in S} \left( 1-\langle \mathfrak{q}\rangle^y x^{\deg(\mathfrak{q})}\right)^{-1}} \\
\ & = \displaystyle{ \prod_{\mathfrak{q}\not\in S} \left( 1- \mathfrak{q}^s\right)^{-1}} \\
\ & = (1-\pr^s) \displaystyle{ \prod_{\mathfrak{q}\neq \infty} \left( 1- \mathfrak{q}^s\right)^{-1}} = (1-\pr^s)\zeta_A(-s). \qedhere
\end{array}  \]
\end{proof}

\begin{rem}
The character $\Psi_y$ can actually be interpreted as a character on ideles via the reciprocity map (composed with some projection map)
$rec_{\calf_\pr}:\mathbb{I}_F/F^* \twoheadrightarrow Gal(\calf_\pr/F)$. This interpretation has been introduced in \cite{ABBL} for $\F(t)$ (and 
generalized in \cite{CPhD}) and exploited not only for characters with values in $\C_\infty$ but also for those with values in 
$\C_{\mathfrak{q}}$ for any prime $\mathfrak{q}$, leading to similar interpolation formulas for $\mathfrak{q}$-adic $L$-functions. 
\end{rem}

\end{section}

\end{document}